\documentclass[preprint,3p,11pt]{elsarticle}

\biboptions{numbers,sort&compress}

\usepackage[T1]{fontenc}
\usepackage[utf8]{inputenc}
\usepackage{amsmath,amssymb,amsthm}
\usepackage{enumitem}
\usepackage{tikz}
\usepackage{graphicx}
\usepackage{hyperref}

\hypersetup{hidelinks}

\newtheorem{lemma}{Lemma}
\newtheorem{theorem}{Theorem}

\tikzset{
  graphnode/.style={circle,draw,fill=white,inner sep=1pt,minimum size=16pt,font=\footnotesize},
  smallgraphnode/.style={circle,draw,fill=white,inner sep=.8pt,minimum size=12pt,font=\scriptsize},
  graphedge/.style={line width=.45pt},
  topcycle/.style={line width=.75pt},
  role/.style={font=\tiny,inner sep=0pt}
}

\begin{document}

\begin{frontmatter}

\title{Subcubic \texorpdfstring{$K_4$}{K4}-minor-free graphs without crumby colorings}
\author[inst1,inst2]{J\'ozsef Pint\'er\corref{cor1}}
\ead{pinterj@edu.bme.hu}
\cortext[cor1]{Corresponding author.}

\affiliation[inst1]{organization={Department of Stochastics, Institute of Mathematics,
  Budapest University of Technology and Economics},
  addressline={Egry J\'ozsef utca 1},
  postcode={1111},
  city={Budapest},
  country={Hungary}}

\affiliation[inst2]{%
  organization={HUN-REN--BME Stochastics Research Group},
  addressline={Egry J\'ozsef utca 1},
  postcode={1111},
  city={Budapest},
  country={Hungary}}

\begin{abstract}
Motivated by Wegner's conjecture on squares of planar graphs, Thomassen conjectured that every 3-connected cubic graph on at least eight vertices admits a red-blue vertex coloring in which the blue subgraph has maximum degree at most 1, while the red subgraph has minimum degree at least 1 and contains no $P_4$. Such colorings are now called crumby colorings. Although this conjecture was disproved in general by Bellitto, Klimo\v{s}ov\'a, Merker, Witkowski and Yuditsky, positive results of Bar\'at, Bl\'azsik and Dam\'asdi led them, in the same subcubic setting, to conjecture that every \texorpdfstring{$K_4$}{K4}-minor-free graph admits a crumby coloring. We disprove this conjecture with a connected subcubic partial 2-tree on 18 vertices. We also disprove its natural 2-connected version with a 2-connected subcubic partial 2-tree on 40 vertices with no crumby coloring. Consequently, the
obstruction to crumby colorability already occurs within treewidth two,
even under 2-connectivity.
\end{abstract}

\begin{keyword}
Crumby coloring \sep Subcubic graph \sep $K_4$-minor-free graph \sep Partial 2-tree \sep Series-parallel graph
\end{keyword}

\end{frontmatter}

\section{Introduction}

A red-blue coloring of the vertices of a graph $G$ is called \emph{crumby} if the subgraph induced by the blue vertices has maximum degree at most 1, while the subgraph induced by the red vertices has minimum degree at least 1 and contains no path on four vertices. Throughout, a $P_4$ means a simple path on four vertices, not necessarily an induced path. Equivalently, the blue vertices induce a disjoint union of isolated vertices and edges, and every red component has no $P_4$ and no isolated vertex.

This coloring condition arose from Thomassen's work on Wegner's conjecture on coloring squares of planar graphs. Wegner conjectured that the square of every planar graph of maximum degree at most 3 is 7-colorable \cite{wegner1977graphs}; this was proved independently by Thomassen \cite{thomassen2018square} and by Hartke, Jahanbekam and Thomas \cite{hartke2016chromatic}. In this context Thomassen formulated a stronger structural conjecture for 3-connected cubic graphs. Bar\'{a}t proved positive results for Generalized Petersen graphs and verified the analogous statement for subcubic trees \cite{barat2019decomposition}. Bellitto, Klimo\v{s}ov\'{a}, Merker, Witkowski and Yuditsky then constructed an infinite family of counterexamples to Thomassen's conjecture \cite{bellitto2021counterexamples}.

Bar\'{a}t, Bl\'{a}zsik and Dam\'{a}sdi subsequently studied which restricted subcubic graph classes still admit crumby colorings \cite{barat2023crumby}. They proved positive results for 2-connected outerplanar graphs, subdivisions of $K_4$, 1-subdivisions of cubic graphs, and genuine subdivisions of subcubic graphs. Since the prototype counterexample of Bellitto et al.\ contains a $K_4$ minor, these results motivated the following conjecture.

\begin{quote}
\textbf{Conjecture 16 of Bar\'{a}t, Bl\'{a}zsik and Dam\'{a}sdi \cite{barat2023crumby}.} Every $K_4$-minor-free graph admits a crumby coloring.
\end{quote}

This conjecture is posed in their study of subcubic graphs. The counterexamples below are subcubic, so they disprove the conjecture both literally and in its intended subcubic setting.

The purpose of this note is to show that the conjecture is false in a strong form. Our first example is connected but not 2-connected. Bar\'at, Bl\'azsik and Dam\'asdi's work also suggests a natural possible escape from such a counterexample. Their proof of the outerplanar case is carried out under a 2-connectivity assumption, and their subsequent discussion identifies the passage from 2-connected blocks to general outerplanar graphs as a source of difficulty. Thus, after a non-2-connected counterexample to their \texorpdfstring{$K_4$}{K4}-minor-free conjecture, it remains natural to ask whether the conjecture survives under an additional 2-connectivity assumption. Our second construction also answers this question in the negative.

\begin{figure}[t]
\centering
\resizebox{.78\linewidth}{!}{%
\begin{tikzpicture}[x=1cm,y=1cm]
  \node[graphnode] (g1) at (0,0.55) {$g_1$};
  \node[graphnode] (h1) at (0,-0.55) {$h_1$};
  \node[graphnode] (e1) at (1.1,1.0) {$e_1$};
  \node[graphnode] (f1) at (1.1,-1.0) {$f_1$};
  \node[graphnode] (c1) at (2.2,1.0) {$c_1$};
  \node[graphnode] (d1) at (2.2,-1.0) {$d_1$};
  \node[graphnode] (a1) at (3.3,1.0) {$a_1$};
  \node[graphnode] (b1) at (3.3,-1.0) {$b_1$};
  \node[graphnode] (x1) at (4.4,0) {$x_1$};

  \node[graphnode] (x2) at (5.35,0) {$x_2$};
  \node[graphnode] (a2) at (6.45,1.0) {$a_2$};
  \node[graphnode] (b2) at (6.45,-1.0) {$b_2$};
  \node[graphnode] (c2) at (7.55,1.0) {$c_2$};
  \node[graphnode] (d2) at (7.55,-1.0) {$d_2$};
  \node[graphnode] (e2) at (8.65,1.0) {$e_2$};
  \node[graphnode] (f2) at (8.65,-1.0) {$f_2$};
  \node[graphnode] (g2) at (9.75,0.55) {$g_2$};
  \node[graphnode] (h2) at (9.75,-0.55) {$h_2$};

  \foreach \u/\v in {x1/a1,x1/b1,a1/c1,b1/d1,c1/d1,c1/e1,d1/f1,e1/g1,e1/h1,f1/g1,f1/h1,
                      x2/a2,x2/b2,a2/c2,b2/d2,c2/d2,c2/e2,d2/f2,e2/g2,e2/h2,f2/g2,f2/h2,x1/x2}
    \draw[graphedge] (\u) -- (\v);
\end{tikzpicture}%
}
\caption{The connected counterexample $G_{18}$, drawn as two copies of the rooted graph $F$ joined by the edge $x_1x_2$.}
\label{fig:g18}
\end{figure}
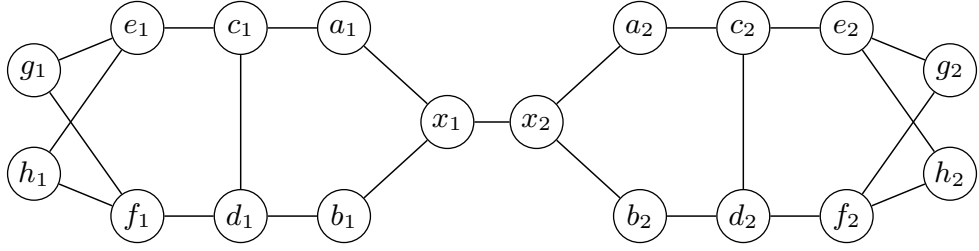

\section{A connected \texorpdfstring{$K_4$}{K4}-minor-free graph without a crumby coloring}

Let $F$ be the graph with vertex set
\[
\{x,a,b,c,d,e,f,g,h\}
\]
and edge set
\[
\{xa,xb,ac,bd,cd,ce,df,eg,eh,fg,fh\}.
\]
Let $G_{18}$ be obtained from two disjoint copies $F_1,F_2$ of $F$ by adding the edge $x_1x_2$. Thus, in the $i$th copy, the vertices are
\[
 x_i,a_i,b_i,c_i,d_i,e_i,f_i,g_i,h_i,
\]
and
\[
E(G_{18})=\bigcup_{i=1}^{2}\{x_ia_i,x_ib_i,a_ic_i,b_id_i,c_id_i,c_ie_i,d_if_i,e_ig_i,e_ih_i,f_ig_i,f_ih_i\}\cup\{x_1x_2\}.
\]

For an illustration of the graph see Figure~\ref{fig:g18}.

We first record the local forcing properties of this gadget: even when the copy of $F$ is placed inside a larger graph, a crumby coloring severely restricts the possible colors near the attachment vertices.

\begin{lemma}
\label{lem:rooted-gadget}
Let $r\in\{a,b\}$. Consider a copy of $F$ inside an ambient graph in which only $x$ and $r$ may have neighbours outside the copy. In any crumby coloring of the ambient graph, the following assertions hold.
\begin{enumerate}[label=(\roman*)]
\item If $x$ is blue, then $r$ is red and has no red neighbour inside the copy of $F$.
\item If $x$ is red, then at least one of $a,b$ is red, even if $x$ has a red neighbour outside the copy of $F$.
\end{enumerate}
\end{lemma}

\begin{proof}
By symmetry it is enough to prove the first assertion for $r=a$. Suppose that $x$ is blue.

First assume that $a$ is blue. Then $b$ is red, since otherwise $x$ would have two blue neighbours, and $c$ is red, since otherwise $a$ would have two blue neighbours. The red vertex $b$ forces $d$ red. If $e$ were red, then $e-c-d-b$ would be a red $P_4$, so $e$ is blue. If $f$ were red, then neither $g$ nor $h$ could be red, because $c-d-f-g$ or $c-d-f-h$ would be a red $P_4$; hence $g,h$ would be blue, giving the blue vertex $e$ two blue neighbours. Thus $f$ is blue. But then a red vertex among $g,h$ would have no red neighbour, while if both $g,h$ are blue then $e$ again has two blue neighbours. This contradiction shows that $a$ is not blue.

Thus $a$ is red. If $c$ were red, then we again get a contradiction. Indeed, if $b$ is red, then $b$ forces $d$ red and $a-c-d-b$ is a red $P_4$. If $b$ is blue, then $d$ is red, because $b$ already has the blue neighbour $x$. Now $f$ cannot be red, since $a-c-d-f$ would be a red $P_4$; hence $f$ is blue. Since $f$ is blue and $d$ is red, at least one of $g,h$ is red; call it $z$. The red vertex $z$ has neighbours only $e$ and $f$, and $f$ is blue, so $e$ is red. But then $d-c-e-z$ is a red $P_4$. Therefore $c$ is blue. The two neighbours of $a$ inside $F$ are $x$ and $c$, both blue, so $a$ has no red neighbour inside $F$. This proves (i).

For (ii), suppose that $x$ is red and has a red neighbour outside $F$, and suppose for a contradiction that both $a,b$ are blue. If both $c,d$ are blue, then $c$ has two blue neighbours $a,d$. If exactly one of $c,d$ is red, then by symmetry let $c$ be red and $d$ blue. The vertex $c$ forces $e$ red. Also $f$ must be red; otherwise the blue vertex $d$ would have two blue neighbours, namely $b$ and $f$. The red vertex $f$ then needs a red neighbour $z$ among $g,h$, and $c-e-z-f$ is a red $P_4$. Finally suppose that $c,d$ are both red. If $e,f$ are both red, then $e-c-d-f$ is a red $P_4$. If $e,f$ are both blue, then no vertex among $g,h$ can be red, so $g,h$ are both blue and $e$ has two blue neighbours. If, say, $e$ is red and $f$ is blue, then at least one of $g,h$ is red, since otherwise the blue vertex $f$ would have two blue neighbours $g,h$; for such a red vertex $z$, the path $d-c-e-z$ is a red $P_4$. The case $e$ blue and $f$ red is symmetric. Hence at least one of $a,b$ is red. This proves (ii).
\end{proof}

\begin{theorem}
\label{thm:g18}
The graph $G_{18}$ is connected, subcubic, $K_4$-minor-free, and has no crumby coloring.
\end{theorem}

\begin{proof}
The graph is visibly connected. In a copy of $F$, the vertices $x,a,b,g,h$ have degree 2 and the vertices $c,d,e,f$ have degree 3. In $G_{18}$ the two vertices $x_1,x_2$ have degree 3, and no degree exceeds 3.

We first show that $G_{18}$ is $K_4$-minor-free. We use the standard
elimination-order certificate for treewidth. Recall that an ordering of the
vertices has width at most $2$ if, when the vertices are eliminated in this
order and missing edges are added among the remaining neighbours of the
eliminated vertex, each eliminated vertex has at most two remaining neighbours.
Such an ordering certifies treewidth at most $2$. Since graphs of treewidth at
most $2$ are partial $2$-trees, equivalently $K_4$-minor-free graphs, it is
enough to find such an ordering.
In one copy of $F$, eliminate the vertices in the order
\[
 g,h,e,f,c,d,a,b,
\]
leaving $x$ until the end. At each step the eliminated vertex has at most two remaining neighbours, after adding the usual fill edge between its remaining neighbours if necessary:
\[
\begin{array}{c|c}
\text{vertex} & \text{remaining neighbours} \\ \hline
 g & e,f \\
 h & e,f \\
 e & c,f \\
 f & c,d \\
 c & a,d \\
 d & a,b \\
 a & x,b \\
 b & x.
\end{array}
\]
Do this in both copies and then eliminate $x_1,x_2$. The edge $x_1x_2$ does not affect the per-copy eliminations, because each $x_i$ is left until all internal vertices of $F_i$ have been eliminated; at that point $x_1$ has only the remaining neighbour $x_2$. This is an elimination ordering of width at most 2. Hence $G_{18}$ has treewidth at most 2, so it has no $K_4$ minor.

It remains to exclude crumby colorings. Apply Lemma~\ref{lem:rooted-gadget}(i) in either copy $F_i$. If $x_i$ were blue, then $a_i$ would be red and would have no red neighbour inside $F_i$; since $a_i$ has no outside neighbour, it would have no red neighbour at all, contradicting the red minimum-degree condition. Thus both $x_1$ and $x_2$ are red. Now $x_i$ has the outside red neighbour $x_{3-i}$, and Lemma~\ref{lem:rooted-gadget}(ii) implies that at least one of $a_i,b_i$ is red. Choose red vertices $p_i\in\{a_i,b_i\}$. Then
\[
 p_1-x_1-x_2-p_2
\]
is a red $P_4$, a contradiction. Thus $G_{18}$ has no crumby coloring.
\end{proof}

Consequently, $G_{18}$ is a connected \texorpdfstring{$K_4$}{K4}-minor-free counterexample to Conjecture~16 of Bar\'at, Bl\'azsik and Dam\'asdi~\cite{barat2023crumby}.

A natural next question is whether every $2$-connected subcubic \texorpdfstring{$K_4$}{K4}-minor-free graph admits a crumby coloring. 

\section{A 2-connected \texorpdfstring{$K_4$}{K4}-minor-free graph without a crumby coloring}

We now give a 2-connected counterexample. The proof uses the same rooted graph $F$ as before, together with one small derived gadget.

Let $R(s,t)$ be the graph obtained from a copy of $F$ with vertices $x,a,b,c,d,e,f,g,h$ by adding the three edges
\[
 sx, \qquad st, \qquad ta.
\]
Thus $R(s,t)$ contains the copy of $F$ on $x,a,b,c,d,e,f,g,h$, with $x$ as the root and $a$ as the distinguished neighbour of the root.

Call a red vertex $v$ \emph{rich in a subgraph} $H$ if there is a red path $v-p-q$ contained in $H$.

\begin{lemma}
\label{lem:rich-gadget}
In any crumby coloring of an ambient graph containing $R(s,t)$, if $s$ is red, then $s$ is rich inside $R(s,t)$, provided that only $s$ and $t$ may have neighbours outside $R(s,t)$.
\end{lemma}

\begin{proof}
Assume that $s$ is red. If $x$ is red, then by Lemma~\ref{lem:rooted-gadget}(ii), at least one of $a,b$ is red (denote it by $z$). Hence there is a red path $s-x-z$ inside $R(s,t)$.

Now suppose that $x$ is blue. By Lemma~\ref{lem:rooted-gadget}(i), applied with $r=a$, the vertex $a$ is red and has no red neighbour inside the copy of $F$. The only neighbour of $a$ outside that copy is $t$. Since $a$ must have a red neighbour, $t$ is red. Hence $s-t-a$ is a red path inside $R(s,t)$.
\end{proof}

Let $G_{40}$ be the graph on vertex set $\{0,1,\ldots,39\}$ with the following 54 edges:
\[
\begin{array}{rrrrrrrr}
0\text{-}1 & 0\text{-}2 & 0\text{-}39 & 1\text{-}3 &
1\text{-}12 & 2\text{-}3 & 2\text{-}4 & 3\text{-}5 \\
4\text{-}6 & 5\text{-}6 & 5\text{-}7 & 6\text{-}8 &
7\text{-}9 & 7\text{-}10 & 8\text{-}9 & 8\text{-}10 \\
11\text{-}12 & 11\text{-}13 & 11\text{-}22 & 12\text{-}14 &
13\text{-}14 & 13\text{-}15 & 14\text{-}16 & 15\text{-}17 \\
16\text{-}17 & 16\text{-}18 & 17\text{-}19 & 18\text{-}20 &
18\text{-}21 & 19\text{-}20 & 19\text{-}21 & 22\text{-}24 \\
22\text{-}30 & 23\text{-}25 & 23\text{-}30 & 24\text{-}25 &
24\text{-}26 & 25\text{-}27 & 26\text{-}28 & 26\text{-}29 \\
27\text{-}28 & 27\text{-}29 & 30\text{-}31 & 31\text{-}32 &
31\text{-}39 & 32\text{-}34 & 33\text{-}34 & 33\text{-}35 \\
33\text{-}39 & 34\text{-}36 & 35\text{-}37 & 35\text{-}38 &
36\text{-}37 & 36\text{-}38 & &
\end{array}
\]
Equivalently, $G_{40}$ consists of the two rich gadgets $R(0,1)$ and $R(11,12)$, a copy of $F$ with root 30 and distinguished root-neighbour 22, a copy of $F$ with root 31 and distinguished root-neighbour 39, and the four connecting edges
\[
1-12, \qquad 11-22, \qquad 30-31, \qquad 39-0.
\]
The two right-hand copies of $F$ are attached as follows:
\[
(x,a,b,c,d,e,f,g,h)=(30,22,23,24,25,26,27,28,29)
\]
and
\[
(x,a,b,c,d,e,f,g,h)=(31,39,32,33,34,35,36,37,38).
\]
The two copies of $F$ inside the rich gadgets $R(0,1)$ and $R(11,12)$ use
\[
(x,a,b,c,d,e,f,g,h)=(2,3,4,5,6,7,8,9,10)
\]
and
\[
(x,a,b,c,d,e,f,g,h)=(13,14,15,16,17,18,19,20,21),
\]
respectively (see Figure~\ref{fig:g40}).

\begin{theorem}
\label{thm:g40}
The graph $G_{40}$ is subcubic, 2-connected, $K_4$-minor-free, and has no crumby coloring.
\end{theorem}

\begin{proof}
The vertices of degree 2 in $G_{40}$ are
\[
4,9,10,15,20,21,23,28,29,32,37,38.
\]
All other vertices have degree 3, so $G_{40}$ is subcubic.

We prove 2-connectivity by an open ear decomposition. Start with the cycle
\[
0,1,12,11,22,30,31,39,0.
\]
Then add the following ears, in the displayed order:
\[
\begin{gathered}
0-2-3-1, \qquad 2-4-6-5-3, \qquad 5-7-9-8-6, \qquad 7-10-8, \\
11-13-14-12, \qquad 13-15-17-16-14, \qquad 16-18-20-19-17, \qquad 18-21-19, \\
30-23-25-24-22, \qquad 24-26-28-27-25, \qquad 26-29-27, \\
31-32-34-33-39, \qquad 33-35-37-36-34, \qquad 35-38-36.
\end{gathered}
\]
Each added path has two distinct endpoints in the graph already constructed, and all its internal vertices are new. Hence this is an open ear decomposition, and $G_{40}$ is 2-connected.

\begin{figure}[t]
\centering
\resizebox{.96\linewidth}{!}{%
\begin{tikzpicture}[x=1cm,y=1cm]
  \node[smallgraphnode] (n0) at (0,0) {$0$};
  \node[smallgraphnode] (n1) at (1.35,0) {$1$};
  \node[smallgraphnode] (n12) at (2.90,0) {$12$};
  \node[smallgraphnode] (n11) at (4.25,0) {$11$};
  \node[smallgraphnode,label={[role]45:$a$}] (n22) at (5.75,0) {$22$};
  \node[smallgraphnode,label={[role]45:$x$}] (n30) at (7.35,0) {$30$};
  \node[smallgraphnode,label={[role]45:$x$}] (n31) at (9.05,0) {$31$};
  \node[smallgraphnode,label={[role]45:$a$}] (n39) at (10.55,0) {$39$};

  \node[smallgraphnode,label={[role]45:$x$}] (n2) at (.42,-.90) {$2$};
  \node[smallgraphnode,label={[role]45:$a$}] (n3) at (1.20,-.90) {$3$};
  \node[smallgraphnode,label={[role]135:$b$}] (n4) at (.05,-1.85) {$4$};
  \node[smallgraphnode,label={[role]45:$c$}] (n5) at (1.70,-1.85) {$5$};
  \node[smallgraphnode,label={[role]225:$d$}] (n6) at (.80,-2.70) {$6$};
  \node[smallgraphnode,label={[role]45:$e$}] (n7) at (1.50,-3.45) {$7$};
  \node[smallgraphnode,label={[role]225:$f$}] (n8) at (.78,-4.35) {$8$};
  \node[smallgraphnode,label={[role]45:$g$}] (n9) at (1.35,-4.75) {$9$};
  \node[smallgraphnode,label={[role]225:$h$}] (n10) at (.7,-5.15) {$10$};

  \node[smallgraphnode,label={[role]45:$a$}] (n14) at (3.00,-.90) {$14$};
  \node[smallgraphnode,label={[role]45:$x$}] (n13) at (4.05,-.90) {$13$};
  \node[smallgraphnode,label={[role]135:$c$}] (n16) at (2.75,-1.85) {$16$};
  \node[smallgraphnode,label={[role]45:$b$}] (n15) at (4.45,-1.85) {$15$};
  \node[smallgraphnode,label={[role]315:$d$}] (n17) at (3.65,-2.70) {$17$};
  \node[smallgraphnode,label={[role]45:$e$}] (n18) at (2.85,-3.45) {$18$};
  \node[smallgraphnode,label={[role]45:$f$}] (n19) at (4.15,-4.35) {$19$};
  \node[smallgraphnode,label={[role]45:$g$}] (n20) at (3.10,-4.75) {$20$};
  \node[smallgraphnode,label={[role]45:$h$}] (n21) at (4.45,-5.15) {$21$};

  \node[smallgraphnode,label={[role]45:$c$}] (n24) at (5.95,-1.15) {$24$};
  \node[smallgraphnode,label={[role]45:$b$}] (n23) at (7.35,-1.15) {$23$};
  \node[smallgraphnode,label={[role]30:$d$}] (n25) at (6.68,-2.00) {$25$};
  \node[smallgraphnode,label={[role]45:$e$}] (n26) at (6.05,-3.00) {$26$};
  \node[smallgraphnode,label={[role]45:$f$}] (n27) at (7.15,-3.10) {$27$};
  \node[smallgraphnode,label={[role]60:$g$}] (n28) at (5.95,-3.95) {$28$};
  \node[smallgraphnode,label={[role]45:$h$}] (n29) at (7.35,-3.95) {$29$};

  \node[smallgraphnode,label={[role]45:$b$}] (n32) at (9.05,-1.15) {$32$};
  \node[smallgraphnode,label={[role]45:$c$}] (n33) at (10.25,-1.15) {$33$};
  \node[smallgraphnode,label={[role]30:$d$}] (n34) at (9.45,-2.00) {$34$};
  \node[smallgraphnode,label={[role]45:$e$}] (n35) at (10.55,-3.00) {$35$};
  \node[smallgraphnode,label={[role]45:$f$}] (n36) at (9.25,-3.10) {$36$};
  \node[smallgraphnode,label={[role]60:$g$}] (n37) at (9.45,-3.95) {$37$};
  \node[smallgraphnode,label={[role]45:$h$}] (n38) at (10.75,-3.95) {$38$};

  \foreach \u/\v in {
    n0/n2,n1/n3,n2/n3,n2/n4,n3/n5,n4/n6,n5/n6,n5/n7,n6/n8,n7/n9,n7/n10,n8/n9,n8/n10,
    n11/n13,n12/n14,n13/n14,n13/n15,n14/n16,n15/n17,n16/n17,n16/n18,n17/n19,n18/n20,n18/n21,n19/n20,n19/n21,
    n30/n23,n22/n24,n23/n25,n24/n25,n24/n26,n25/n27,n26/n28,n26/n29,n27/n28,n27/n29,
    n31/n32,n39/n33,n32/n34,n33/n34,n33/n35,n34/n36,n35/n37,n35/n38,n36/n37,n36/n38}
    \draw[graphedge] (\u) -- (\v);

  \foreach \u/\v in {n0/n1,n1/n12,n12/n11,n11/n22,n22/n30,n30/n31,n31/n39}
    \draw[topcycle] (\u) -- (\v);
  \draw[topcycle] (n39) .. controls (8.7,3.00) and (1.7,3.00) .. (n0);

  \node[font=\scriptsize] at (.85,-5.75) {$R(0,1)$};
  \node[font=\scriptsize] at (3.55,-5.75) {$R(11,12)$};
  \node[font=\scriptsize] at (6.70,-4.65) {copy of $F$, root 30};
  \node[font=\scriptsize] at (9.95,-4.65) {copy of $F$, root 31};
\end{tikzpicture}%
}
\caption{The 2-connected counterexample $G_{40}$. Small labels mark the roles $x,a,b,c,d,e,f,g,h$ inside the four copies of $F$.}
\label{fig:g40}
\end{figure}
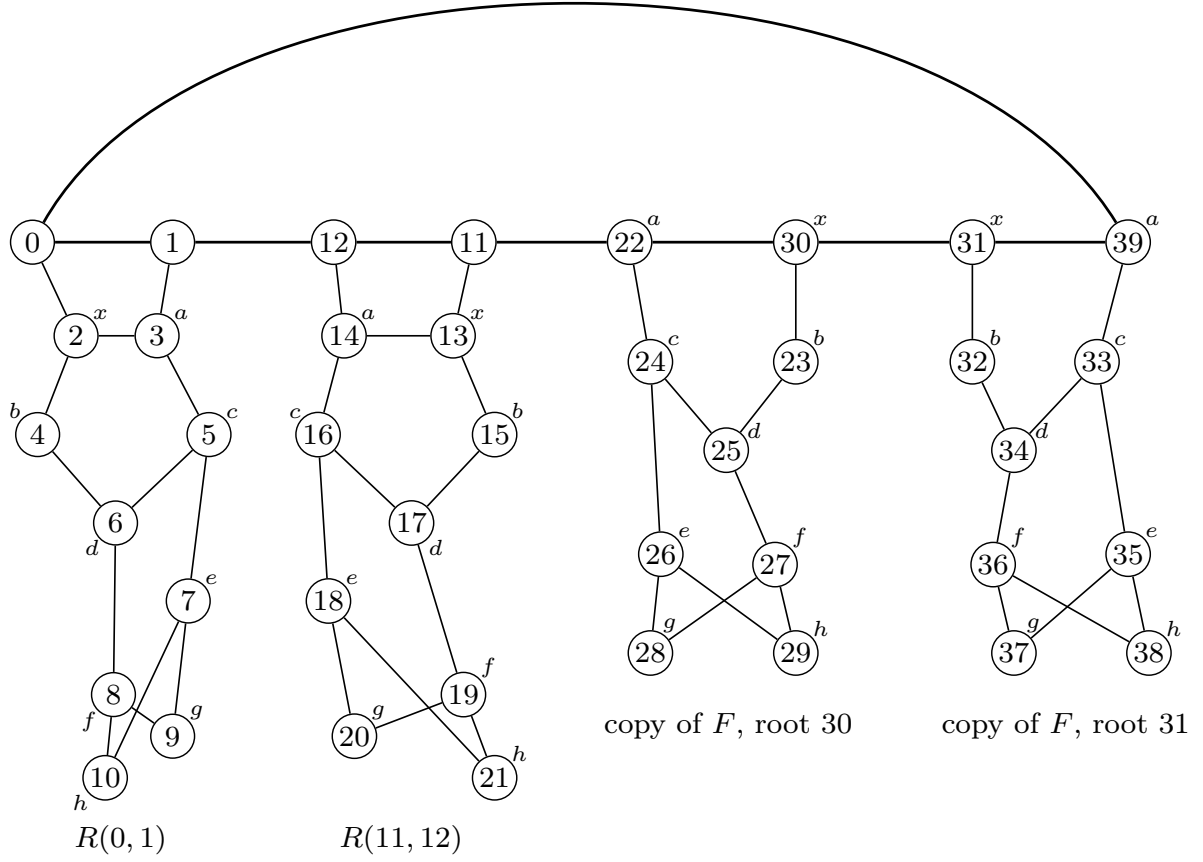

For $K_4$-minor-freeness, let $E$ denote a single two-terminal edge and let $\cdot$ and $\parallel$ denote series and parallel composition. The rooted graph $F$, viewed as a two-terminal graph with terminals $x$ and $a$, is series-parallel; for example, it is the graph
\[
 Q=E\parallel(P_2\cdot A\cdot E), \qquad P_2=E\cdot E, \qquad A=E\parallel(E\cdot(P_2\parallel P_2)\cdot E).
\]
Here $Q$ is the two-terminal copy of $F$ whose first terminal is the root $x$ and whose second terminal is the distinguished neighbour $a$. The rich gadget $R(s,t)$ is the $s$-to-$t$ graph
\[
 R=E\parallel(E\cdot Q\cdot E),
\]
where the parallel edge is the terminal edge $st$, and the other branch is the edge $sx$, followed by $Q$ from $x$ to $a$, followed by the edge $at$. Thus all four modules in Figure~\ref{fig:g40} are two-terminal series-parallel networks.

Let $X^{\mathrm{rev}}$ denote the same two-terminal network $X$ with its terminal order reversed. The graph obtained by deleting the edge $0-39$ is, with the terminal labels read along the top path $0,1,12,11,22,30,31,39$, the series composition
\[
 R\cdot E\cdot R^{\mathrm{rev}}\cdot E\cdot Q^{\mathrm{rev}}\cdot E\cdot Q.
\]
Here the first $R$ is $R(0,1)$, the reversed copy $R^{\mathrm{rev}}$ is $R(11,12)$ read from 12 to 11, the $Q^{\mathrm{rev}}$ copy is the $F$-module read from 22 to its root 30, and the final $Q$ copy is the $F$-module read from its root 31 to 39. Reversing the terminal order of a two-terminal series-parallel network preserves series-parallelness, and adding the edge $0-39$ is one final parallel composition with $E$. Therefore $G_{40}$ is series-parallel, hence has treewidth at most 2, and so contains no $K_4$ minor.

It remains to exclude crumby colorings. Assume, for contradiction, that $G_{40}$ has one.

First we show that 30 is red. If 30 were blue, then Lemma~\ref{lem:rooted-gadget}(i), applied to the copy of $F$ with root 30 and distinguished neighbour 22, would imply that 22 is red and has no red neighbour inside that copy of $F$. The only neighbour of 22 outside the copy is 11, so 11 must be red. By Lemma~\ref{lem:rich-gadget}, the red vertex 11 is rich inside $R(11,12)$; say $11-u-v$ is a red path there. Then $22-11-u-v$ is a red $P_4$, a contradiction. Hence 30 is red.

Next, 31 is not red. Indeed, if 31 were red, then Lemma~\ref{lem:rooted-gadget}(ii) would give red neighbours $u$ of 30 and $v$ of 31 inside their respective copies of $F$. Then $u-30-31-v$ would be a red $P_4$, contradiction. Thus 31 is blue.

Now apply Lemma~\ref{lem:rooted-gadget}(i) to the copy of $F$ with root 31 and distinguished neighbour 39. Since 31 is blue, the vertex 39 is red and has no red neighbour inside that copy of $F$. Its only neighbour outside the copy is 0, so 0 must be red. By Lemma~\ref{lem:rich-gadget}, the red vertex 0 is rich inside $R(0,1)$; say $0-u-v$ is a red path there. The edge $39-0$ gives the red $P_4$
\[
39-0-u-v,
\]
a contradiction. Therefore $G_{40}$ has no crumby coloring.
\end{proof}

Thus the obstruction shown above is not merely a consequence of cut vertices: even within the 2-connected \texorpdfstring{$K_4$}{K4}-minor-free subcubic graphs, crumby colorings need not exist. In particular, $G_{40}$ disproves the natural 2-connected strengthening of Conjecture~16 of Bar\'at, Bl\'azsik and Dam\'asdi~\cite{barat2023crumby}.

\section{Concluding remarks}

The graph $G_{18}$ disproves Conjecture 16 exactly as stated. The graph $G_{40}$ shows that the natural repair obtained by imposing 2-connectivity is also false. Both graphs are subcubic partial 2-trees, so they contain no $K_4$ minor. The proof for $G_{40}$ uses only repeated applications of the same rooted forcing lemma as the 18-vertex example.

We do not know whether 18 is the minimum order of a connected counterexample, or whether 40 is optimal among 2-connected subcubic $K_4$-minor-free counterexamples. Since $K_4$-minor-free graphs have treewidth at most 2, there is no nontrivial 3-connected version of the question. Finally, the examples here do not address the bipartite or outerplanar variants discussed in \cite{barat2023crumby}: the gadget $F$ contains the 5-cycle $c-e-g-f-d-c$, and contracting the edge $cd$ gives a $K_{2,3}$ minor. Thus both $G_{18}$ and $G_{40}$ are non-bipartite and non-outerplanar.

\section*{Funding}
J\'ozsef Pint\'er is funded by Project No. KDP-IKT-2023-900-I1-00000957/0000003 with support provided by the Ministry of Culture and Innovation of Hungary from the National Research, Development and Innovation Fund, financed under the KDP-2023 funding scheme. The funder had no role in the design of the study, the preparation of the manuscript, or the decision to submit the article for publication.

\section*{Declaration of competing interest}
The author declares that he has no competing interests.

\section*{Data availability}
No datasets were generated or analyzed during the current study.

\section*{Declaration of generative AI and AI-assisted technologies in the manuscript preparation process}
During the preparation of this work, the author used OpenAI's ChatGPT for brainstorming, organization, \LaTeX{} drafting, and editorial assistance. The author reviewed and edited the output and takes full responsibility for the content of the manuscript.

\section*{Acknowledgements}
The author thanks J. Bar\'at, Z. L. Bl\'azsik, and G. Dam\'asdi for their work on crumby colorings and for formulating the $K_4$-minor-free conjecture resolved here.

\bibliographystyle{abbrvurl}
\bibliography{references}

\end{document}